\documentclass[12pt,a4paper]{article}
\usepackage{fontspec}
\usepackage[margin=27mm,top=25mm,bottom=25mm,headheight=15pt,headsep=8mm]{geometry}
\usepackage{amsmath,amssymb,amsthm,mathtools}
\usepackage{enumitem,fancyhdr,needspace,titlesec}
\usepackage{float}
\usepackage{booktabs}
\usepackage{microtype}
\usepackage{tikz}
\usetikzlibrary{arrows.meta,decorations.pathreplacing}
\colorlet{paperlightgray}{black!8}
\colorlet{papermediumgray}{black!35}
\tikzset{
  paper vertex/.style={circle,draw=black,fill=black,inner sep=0pt,minimum size=3.5pt},
  paper root/.style={circle,draw=black,fill=white,inner sep=0pt,minimum size=6pt,line width=.7pt},
  paper marked/.style={rectangle,draw=black,fill=black,inner sep=0pt,minimum size=4pt},
  paper outline/.style={draw=black!55,dashed,line width=.5pt},
  paper region fill/.style={fill=paperlightgray,draw=none},
  paper region/.style={paper outline,fill=paperlightgray},
  paper edge/.style={draw=black,line width=.6pt},
  paper brace/.style={draw=black,line width=.4pt,decorate,decoration={brace}}
}
\newcommand{\paperunderbrace}[1]{%
  \mathop{%
    \begin{tikzpicture}[baseline=(expression.base)]
      \node[inner sep=0pt,outer sep=0pt] (expression) {$\displaystyle #1$};
      \draw[paper brace,decoration={brace,mirror,amplitude=4pt}]
        ([yshift=-2pt]expression.south west)--([yshift=-2pt]expression.south east);
    \end{tikzpicture}%
  }\limits
}
\tikzset{
  paper B region/.style={fill=papermediumgray,draw=none},
  paper vertex group/.style={paper outline},
  paper component frame/.style={paper outline,solid,rounded corners=5pt},
  paper near-left y/.initial=1.22,
  paper near-right y/.initial=1.05,
  pics/paper root branch/.style args={#1/#2}{code={
    \path[paper region fill] (0,1.05) ellipse (.62cm and .58cm);
    \path[paper region fill] (0,-.65) ellipse (.77cm and .56cm);
    \ifnum#1=2\relax
      \path[paper B region] (0,1.05) ellipse (.62cm and .58cm);
    \fi
    \ifnum#1=1\relax
      \begin{scope}
        \clip (0,1.05) ellipse (.62cm and .58cm);
        \path[paper B region]
          (-.75,.35)--(-.75,1.75)--(.08,1.75)
          .. controls (-.2,1.35) and (.28,.85) .. (.02,.35)--cycle;
      \end{scope}
    \fi
    \ifnum#2=1\relax
      \begin{scope}
        \clip (0,-.65) ellipse (.77cm and .56cm);
        \path[paper B region]
          (-.85,-1.3)--(-.85,0)--(-.08,0)
          .. controls (-.25,-.39) and (.51,-.79) .. (.34,-1.3)--cycle;
      \end{scope}
    \fi
    \draw[paper vertex group] (0,1.05) ellipse (.62cm and .58cm);
    \draw[paper vertex group] (0,-.65) ellipse (.77cm and .56cm);
    \coordinate (-near-left) at (-.23,\pgfkeysvalueof{/tikz/paper near-left y});
    \coordinate (-near-right) at (.25,\pgfkeysvalueof{/tikz/paper near-right y});
    \coordinate (-far-left) at (-.24,-.54);
    \coordinate (-far-right) at (.30,-.75);
    \draw[paper edge] (-near-left)--(-far-left);
    \draw[paper edge] (-near-right)--(-far-right);
    \foreach \vertexname in {near-left,near-right,far-left,far-right}
      \node[paper vertex] at (-\vertexname) {};
  }},
  pics/paper other component/.style={code={
    \path[paper region fill] (0,0) ellipse (.75cm and .58cm);
    \begin{scope}
      \clip (0,0) ellipse (.75cm and .58cm);
      \path[paper B region]
        (-.85,-.65)--(.85,-.65)--(.85,-.13)
        .. controls (.3,.15) and (-.2,-.36) .. (-.85,-.06)--cycle;
    \end{scope}
    \draw[paper vertex group] (0,0) ellipse (.75cm and .58cm);
    \node at (0,.13) {$\cdots$};
  }},
  pics/paper lower components/.style={code={
    \pic at (-2.9,-2.5) {paper other component};
    \pic at (-.5,-2.5) {paper other component};
    \node at (-2.9,-3.3) {$C_2$};
    \node at (-.5,-3.3) {$C_3$};
    \node at (1.4,-2.5) {$\cdots$};
  }},
  pics/paper B legend/.style args={#1}{code={
    \path[paper B region] (2.2,-2.65) rectangle (2.65,-2.35);
    \node[anchor=west] at (2.8,-2.5) {vertices in $#1$};
  }}
}
\usepackage[hidelinks,unicode,pdfencoding=auto]{hyperref}
\usepackage{xurl}
\hypersetup{
  pdftitle={A proof of Vince's conjecture on the average order of connected induced subgraphs},
  pdfauthor={Jungang Chen and Hongxia Ma},
  pdfsubject={The minimum degree three conjecture, graphs with cut vertices, fixed minimum degree, and joins},
  pdfkeywords={average order, connected induced subgraph, extremal graph theory}
}
\titleformat{\section}{\large\bfseries}{\thesection}{1em}{}
\titleformat{\subsection}{\normalsize\bfseries}{\thesubsection}{1em}{}
\titlespacing*{\section}{0pt}{2.2ex plus .4ex minus .2ex}{1.2ex plus .2ex}
\titlespacing*{\subsection}{0pt}{1.8ex plus .3ex minus .2ex}{.8ex plus .2ex}
\newtheoremstyle{cnplain}{7pt}{7pt}{\normalfont}{0pt}{\bfseries}{.}{.6em}{}
\theoremstyle{cnplain}
\newtheorem{theorem}{Theorem}[section]
\newtheorem{lemma}[theorem]{Lemma}
\newtheorem{proposition}[theorem]{Proposition}
\newtheorem{corollary}[theorem]{Corollary}
\newtheorem*{conjecture}{Conjecture}

\numberwithin{equation}{section}
\setlist[enumerate]{label=\textnormal{(\roman*)},leftmargin=2.5em,itemsep=2pt,topsep=4pt}
\allowdisplaybreaks[1]
\newcommand{\CC}{\mathcal C}
\newcommand{\QQ}{\mathcal Q}
\newcommand{\DD}{\mathcal D}
\newcommand{\calA}{\mathcal A}
\newcommand{\doi}[1]{\href{https://doi.org/#1}{\nolinkurl{#1}}}
\newcommand{\ord}[1]{v(#1)}

\title{A proof of Vince's conjecture on the average order of connected induced subgraphs}
\author{%
  \begin{tabular}[t]{@{}c@{\hspace{3em}}c@{}}
    Jungang Chen & Hongxia Ma\thanks{Corresponding author.}\\[5pt]
    {\small School of Mathematical Sciences} & {\small Xinjiang Normal University}\\
    {\small Xiamen University, P. R. China} & {\small P. R. China}\\[4pt]
    {\small\href{mailto:jgchen@stu.xmu.edu.cn}{\texttt{jgchen@stu.xmu.edu.cn}}} &
    {\small\href{mailto:hongxiama516@163.com}{\texttt{hongxiama516@163.com}}}
  \end{tabular}%
}
\date{}
\makeatletter
\renewcommand{\@maketitle}{%
  \newpage\null\vskip .5em
  \begin{center}
    {\LARGE\bfseries\@title\par}%
    \ifx\@author\@empty\else
      \vskip 1em
      {\normalsize\lineskip .5em
       \begin{tabular}[t]{c}\@author\end{tabular}\par}%
    \fi
    \ifx\@date\@empty\else
      \vskip .5em {\normalsize\@date\par}%
    \fi
  \end{center}
  \vskip .75em
}
\makeatother

\begin{document}
\maketitle
\begin{abstract}
We prove that the average order of connected induced subgraphs of a connected simple graph of order $n$ and minimum degree at least $3$ is strictly greater than $n/2$, which settles a conjecture of Vince~[\emph{J. Graph Theory} 97 (2021), 82--103].
We go further and show that, for connected simple graphs of order $n$ and minimum degree exactly $k\ge3$, $K_k\vee(K_1\cup K_{n-k-1})$ uniquely minimizes this average under additional conditions.
\end{abstract}

\section{Introduction}\label{sec:intro}

Connectedness is a fundamental topic in graph theory, and connected induced subgraphs play an important role in the study of graph structure.
We study the average order of connected induced subgraphs, which was introduced for general graphs by Kroeker, Mol, and Oellermann~\cite{KroekerMolOellermann2018}.
For a tree, this is the average order of its subtrees, first studied by Jamison~\cite{Jamison1983}.
For a general graph, several subtrees may have the same vertex set, so the two types of averages need not agree.
A natural extremal question is how small this average can be among connected graphs of a given order.

All graphs in this paper are finite and simple.
For a graph $G$ of order $n$, write $V(G)$ and $E(G)$ for its vertex and edge sets, and $v(G)$ for its order.
For $X\subseteq V(G)$, let $G[X]$ be the subgraph induced by $X$ and let $G-X$ be the graph obtained by deleting $X$.
A nonempty vertex set is called connected if it induces a connected subgraph. Define
\begin{equation*}
 \CC(G)=\{\varnothing\ne U\subseteq V(G):G[U]\text{ is connected}\}.
\end{equation*}
Write $N(G)=|\CC(G)|$ and $S(G)=\sum_{U\in\CC(G)}|U|$.
For a nonempty graph, the average order and density of its connected induced subgraphs are
$A(G)=S(G)/N(G)$ and $D(G)=A(G)/v(G)$, respectively.
We write $\Gamma_G(v)$ for the neighborhood of $v$, $d_G(v)=|\Gamma_G(v)|$ for its degree, and $\delta(G)$ for the minimum degree.

Haslegrave~\cite[Theorem 1]{HaslegravePath2022} and Vince~\cite[Theorem 1.1]{VinceLower2022} independently proved that
$A(G)\ge(n+2)/3$ for every connected graph of order $n$, with equality if and only if the graph is a path.
The path has vertices of small degree, and it is natural to ask whether excluding such vertices forces a larger average.
Vince~\cite[Conjecture 7]{VinceProblems2021} conjectured that every connected graph of order $n$ with minimum degree at least $3$ satisfies $A(G)\ge n/2$.
He restated this conjecture in~\cite[Conjecture 1]{VinceUIS2025}. Later, several partial results support this conjecture: Haslegrave~\cite[Theorem 3.7]{HaslegraveDegree2022} proved that $D(G)\to1/2$ along any sequence of connected graphs whose minimum degree tends to infinity.
Vince~\cite[Theorem 2]{VinceConnectivity2021} proved that every $t$-connected graph of order $n$ satisfies
$A(G)\ge\frac n2\left(1-\frac1{2^t+1}\right)$, which gives $A(G)\ge2n/5$ for 2-connected graphs.
More recently, Vince~\cite[Theorem 1.4(2)]{VinceUIS2025} proved the conjecture for connected vertex-transitive graphs by using union-intersection systems. However, these results do not settle the conjecture in full.

Haslegrave~\cite[Section 3]{HaslegravePath2022} explained an obstacle to induction for the conjecture: deleting a vertex can destroy the minimum degree condition. We overcome this difficulty by extending Haslegrave's work on $(S,H)$-connected sets to obtain stronger rooted counting estimates and constructing smaller graphs with the required minimum degree.
Define
\begin{equation}\label{eq:F}
 F(G)=2S(G)-v(G)N(G)
     =\sum_{U\in\CC(G)}(2|U|-v(G)).
\end{equation}
Here we set $N,S,F$ to zero for the empty graph.
Since $A(G)=n/2+F(G)/(2N(G))$, the conjecture is equivalent to $F(G)\ge0$.
We prove the stronger bound $F(G)\ge n$, which is sharp because complete graphs satisfy
\begin{equation}\label{eq:complete}
 N(K_n)=2^n-1,\qquad S(K_n)=n2^{n-1},\qquad F(K_n)=n.
\end{equation}

Following~\cite{BondyMurty2008}, a graph is 2-connected if it is connected, has at least three vertices, and has no cut vertex. We use connected sets with connected complements to find an edge whose contraction preserves 2-connectivity and allows induction on the order (Section~\ref{sec:biconnected}). For graphs with minimum degree at least three and a cut vertex, we strengthen Haslegrave's estimates by requiring degree at least three only away from a root and its neighbors.
We split at a cut vertex or cut edge and add to each rooted part a vertex adjacent to the root and all its neighbors.
The resulting graphs have smaller order and minimum degree at least three, allowing the strengthened estimates to complete the induction (Section~\ref{sec:mainproof}).
For the equality cases, let $G_{\mathrm c}$ and $G_{\mathrm b}$ be obtained from two disjoint copies of $K_4$ by identifying one vertex from each copy or adding an edge between the copies, respectively.
For minimum degree exactly $k\ge3$, write $\cup$ for disjoint union and $\vee$ for the join obtained by adding all edges between the parts.
Put $m=n-k-1\ge0$ and define
\begin{equation*}
 H_{n,k}=K_k\vee(K_1\cup K_m),\qquad \mu_{n,k}=\frac n2+
  \frac{2+(k-1)2^m}{2(2^n-2^m)}.
\end{equation*}
Here $K_0$ is the empty graph and $\mu_{n,k}=A(H_{n,k})$ by Lemma~\ref{lem:candidate-count}.
Following the edge-addition comparison for cographs of Kroeker, Mol, and Oellermann~\cite{KroekerMolOellermann2018}, we preserve the neighborhood of a vertex of minimum degree and add all missing edges among the other vertices.
Estimating the average order of the newly connected induced subgraphs, together with the cut-vertex bounds, gives the fixed-degree results (Section~\ref{sec:fixeddegree}).

Our main results are as follows.

\begin{theorem}\label{thm:biconnected}
Every 2-connected simple graph $G$ of order $n$ satisfies $F(G)\ge n$.
\end{theorem}

\begin{theorem}\label{thm:main}
Let $G$ be a connected simple graph of order $n$ with $\delta(G)\ge3$. Then
\begin{equation*}
 F(G)\ge n,
\end{equation*}
with equality if and only if $G\cong K_n$. If $G$ has a cut vertex, then
\begin{equation*}
 F(G)\ge2n,
\end{equation*}
with equality if and only if $G\cong G_{\mathrm c}$ or $G\cong G_{\mathrm b}$.
\end{theorem}

\begin{theorem}\label{thm:cutaverage}
Let $G$ be connected with $\delta(G)\ge3$, and write $n=v(G)$.
\begin{enumerate}
 \item If $G$ has a cut vertex but no cut edge, then $A(G)\ge n/2+7/78$,
 with equality if and only if $G\cong G_{\mathrm c}$.
 \item If $G$ has a cut edge, then $A(G)\ge n/2+4/47$,
 with equality if and only if $G\cong G_{\mathrm b}$.
\end{enumerate}
\end{theorem}

\begin{theorem}\label{thm:fixeddegree}
Let $G$ be a connected simple graph of order $n$ with $\delta(G)=k\ge3$. If any of the following conditions holds, then
\begin{equation*}
 A(G)\ge\mu_{n,k}.
\end{equation*}
\begin{enumerate}
 \item $G$ has a cut vertex;
 \item $G=G[A]\vee G[B]$, where $V(G)=A\mathbin{\dot\cup}B$ and $|A|,|B|\ge2$;
 \item $3n\le4(k+1)$.
\end{enumerate}
Within the union of these classes, equality holds if and only if $G\cong H_{n,k}$; under condition \textnormal{(i)}, the inequality is always strict.
Under condition \textnormal{(ii)}, the stronger counting comparison $F(G)\ge F(H_{n,k})$ also holds.
\end{theorem}

\section{Lower bounds for \texorpdfstring{$F(G)$}{F(G)} in 2-connected graphs}\label{sec:biconnected}

For a graph $G$ and a nonempty set $X\subseteq V(G)$, we use the notations of Vince~\cite[Section 3.1]{VinceUIS2025}:
$N(G,X)=|\{U\in\CC(G):X\subseteq U\}|$ and $\overline N(G,X)=|\{U\in\CC(G):U\cap X=\varnothing\}|=N(G-X)$. In particular, when $X=\{v\}$, write $N(G,v)$ and $\overline N(G,v)$; then
$N(G)=N(G,v)+\overline N(G,v)$.
Following Vince~\cite[Section 8]{VinceProblems2021}, we consider nonempty connected sets whose complements are also nonempty and connected.
The following lemma will help us keep both a set and its complement connected.

\begin{lemma}\label{lem:twoboundary}
Let $G$ be connected with at least two vertices, and let $D\subseteq V(G)$.
Suppose that, for every $v\in V(G)$, every component of $G-v$ meets $D$.
Then $D$ contains at least two distinct vertices that are not cut vertices.
\end{lemma}
\begin{proof}
Clearly $|D|\ge 2$.  Fix any $w\in V(G)$. We show that $D\setminus\{w\}$ contains a vertex that is not a cut vertex.
The hypothesis applied to $G-w$ ensures that $D\setminus\{w\}$ is nonempty.
Choose $u\in D\setminus\{w\}$ so that the component $C$ of $G-u$ containing $w$ has maximum order.
If $u$ is a cut vertex, take another component $C'$ of $G-u$ and choose $v\in D\cap V(C')$.
Since $G$ is connected, $u$ has a neighbor in $C$.
Thus $V(C)\cup\{u\}$ induces a connected subgraph of $G-v$, so the component of $G-v$ containing $w$ has more vertices than $C$.
This contradicts the choice of $u$, therefore $u$ is not a cut vertex.
Applying the same argument with $w=u$ gives a second such vertex in $D$, distinct from $u$.
\end{proof}

The next lemma provides the edge needed for the contraction argument below.

\begin{lemma}\label{lem:descent}
Let $G$ be a 2-connected graph. Then there exists an edge $uv\in E(G)$ such that
$G-\{u,v\}$ is connected and $N(G,\{u,v\})\ge\overline N(G,\{u,v\})+1$.
\end{lemma}
\begin{proof}
Since $G$ is 2-connected, then $G[V(G)\setminus v]$ and $G-(V(G)\setminus v)$ are all connected and $2=N(G,V(G)\setminus v)=\overline N(G,V(G)\setminus v)+1$. So we can choose $U\subsetneq V(G)$ of minimum size such that $|U|\ge2$, both $G[U]$ and $G-U$
are connected, and $N(G,U)\ge\overline N(G,U)+1$.

Suppose that $|U|>2$, and let $D$ consist of the vertices of $U$ adjacent to $V(G)\setminus U$.
For every $u\in U$, each component of $G[U]-u$ meets $D$; otherwise, $G-u$ would be disconnected.
By Lemma~\ref{lem:twoboundary}, there are distinct $v_1,v_2\in D$ such that
$U'_i=U\setminus\{v_i\}$ induces a connected graph. Since $v_i\in D$, $G-U'_i$ is also connected.

For $i=1,2$, define
$\mathcal A_i=\{U_0\in\CC(G):U_0\cap U=\{v_i\}\}$ and
$\mathcal B_i=\{U_0\in\CC(G):U_0\cap U=U'_i\}$.
For distinct $i,j\in\{1,2\}$, consider the map
\[
 \begin{aligned}
 \varphi_{ij}\colon\mathcal A_i&\longrightarrow\mathcal B_j,\\
 U_0&\longmapsto U_0\cup U'_j.
 \end{aligned}
\]
For $U_0\in\mathcal A_i$, the union $U_0\cup U'_j$ is connected because both sets contain $v_i$,
and it meets $U$ exactly in $U'_j$. Hence it belongs to $\mathcal B_j$, so $\varphi_{ij}$ is well-defined.
For $U_1,U_2\in\mathcal A_i$, $U_1\cup U'_j=U_2\cup U'_j$ simply implies $U_1=U_2$, so $\varphi_{ij}$ is injective. Hence $|\mathcal B_1|\ge|\mathcal A_2|,|\mathcal B_2|\ge|\mathcal A_1|$ yields $|\mathcal B_1|+|\mathcal B_2|\ge|\mathcal A_1|+|\mathcal A_2|$,
then $|\mathcal B_i|\ge|\mathcal A_i|$ for some $i$.
For this $i$:
\[
 N(G,U'_i)-\overline N(G,U'_i)
 =N(G,U)-\overline N(G,U)+|\mathcal B_i|-|\mathcal A_i|\ge1,
\]
contradicting the minimality of $U$.
Thus $U=\{u,v\}$, and $uv$ is the required edge.
\end{proof}

To compare $F$ before and after contraction, we also need connected sets containing a fixed root.
For a graph $H$ rooted at $r$, define
$S(H,r)=\sum_{\substack{r\in U\in\CC(H)}}|U|$.
Also put
$F(H,r)=2S(H,r)-(\ord{H}+1)N(H,r)$.

\begin{corollary}\label{cor:roothalf}
Every connected rooted graph $(H,r)$ satisfies
$F(H,r)\ge0$.
\end{corollary}
\begin{proof}
Vince~\cite[Corollary 3.2]{VinceLower2022} proved that
$
 \frac{S(H,r)}{N(H,r)}\ge\frac{\ord{H}+1}{2}\iff F(H,r)\ge 0.\qedhere
$
\end{proof}

For edge contraction, we follow Bondy and Murty~\cite[Sections 2.2 and 2.3]{BondyMurty2008}. For an edge $uv\in E(G)$,
let $G/uv$ denote the graph obtained by deleting $uv$ and identifying its ends, keeping any parallel edges that arise.
Let $G_{uv}$ denote the underlying simple graph of $G/uv$.
If $uv$ is a cut edge, no parallel edges arise in $G/uv$, so it equals $G_{uv}$.

The next identity gives the change in $F$ under edge contraction.

\begin{lemma}\label{lem:contraction}
If $uv\in E(G)$ and the graphs $G-u,G-v$ are both connected, then
\begin{equation*}
 F(G)-F(G_{uv})=N(G,\{u,v\})-\overline N(G,\{u,v\})+F(G-v,u)+F(G-u,v).
\end{equation*}
\end{lemma}
\begin{proof}
Let $w$ be the contracted vertex and put $n=\ord{G}$.
In the diagram below, it shows the partitions of $\CC(G)$ and $\CC(G_{uv})$ respectively.
The bijection $\varphi\colon\CC_{uv}\to\CC_w$ sends $U$ to $(U\setminus\{u,v\})\cup\{w\}$;
its inverse replaces $w$ with $u,v$, preserving connectivity because $uv\in E(G)$.
The bijection $\psi\colon\CC_{\bar u\bar v}\to\CC_{\bar w}$ is the identity, since $G-\{u,v\}=G_{uv}-w$.
Thus $\varphi$ decreases the order of a set by one, and $\psi$ preserves it.
\begin{center}
\begin{tikzpicture}[font=\small,>=Latex,every node/.style={inner sep=0pt}]
 \node[anchor=west] at (0,0) {$\CC_{uv}:\quad \{u,v\}\subseteq U$};
 \node[anchor=west] at (0,-0.85) {$\CC_{\bar u\bar v}:\quad U\cap\{u,v\}=\varnothing$};
 \node[anchor=west] at (0,-1.70) {$\CC_{u\bar v}:\quad u\in U,\ v\notin U$};
 \node[anchor=west] at (0,-2.55) {$\CC_{\bar u v}:\quad u\notin U,\ v\in U$};
 \node[anchor=west] at (6.4,0) {$\CC_w:\quad w\in U$};
 \node[anchor=west] at (6.4,-0.85) {$\CC_{\bar w}:\quad w\notin U$};
 \draw[<->,line width=0.45pt] (4.5,0) -- node[above=2pt] {$\varphi$} (6.1,0);
 \draw[<->,line width=0.45pt] (4.5,-0.85) -- node[above=2pt] {$\psi$} (6.1,-0.85);
 \draw[paper brace,decoration={brace,amplitude=4pt}]
  (-0.22,-2.88) -- (-0.22,0.33);
 \node[anchor=east] at (-0.55,-1.275) {$\CC(G)$};
 \draw[paper brace,decoration={brace,amplitude=4pt}]
  (8.8,0.33) -- (8.8,-1.18);
 \node[anchor=west] at (9.15,-0.425) {$\CC(G_{uv})$};
\end{tikzpicture}
\end{center}

Since $\ord{G-v}+1=\ord{G-u}+1=n$, the left two classes give
\begin{align*}
 \sum_{U\in\CC_{u\bar v}}(2|U|-n)
 &=2S(G-v,u)-nN(G-v,u)=F(G-v,u),\\
 \sum_{U\in\CC_{\bar u v}}(2|U|-n)
 &=2S(G-u,v)-nN(G-u,v)=F(G-u,v).
\end{align*}
The bijections $\varphi$ and $\psi$, together with~\eqref{eq:F}, give
\begin{align*}
 F(G)-F(G_{uv})
 &=\sum_{U\in\CC_{uv}}
   \paperunderbrace{\bigl((2|U|-n)-(2(|U|-1)-(n-1))\bigr)}_{1}\\
 &\quad+\sum_{U\in\CC_{\bar u\bar v}}
   \paperunderbrace{\bigl((2|U|-n)-(2|U|-(n-1))\bigr)}_{-1}\\
 &\quad+F(G-v,u)+F(G-u,v)\\
 &=|\CC_{uv}|-|\CC_{\bar u\bar v}|+F(G-v,u)+F(G-u,v)\\
 &=N(G,\{u,v\})-\overline N(G,\{u,v\})+F(G-v,u)+F(G-u,v).\qedhere
\end{align*}
\end{proof}

\begin{proof}[Proof of Theorem~\ref{thm:biconnected}]
We proceed by induction on $n$. If $n=3$, then $G=K_3$ and the conclusion follows from~\eqref{eq:complete}.
Suppose that $n\ge4$, and choose the edge $uv$ given by Lemma~\ref{lem:descent}.

Let $w$ be the contracted vertex of $G_{uv}$.
Deleting $w$ from $G_{uv}$ gives the connected graph $G-\{u,v\}$.
Deleting any other vertex $w'$ gives $G_{uv}-w'=(G-w')_{uv}$,
which is again connected. Thus $G_{uv}$ is a 2-connected simple graph of order $n-1$.
By the induction hypothesis, Lemma~\ref{lem:contraction}, and Corollary~\ref{cor:roothalf},
\[
 F(G)\ge F(G_{uv})+1\ge(n-1)+1=n.
\]
\end{proof}

\section{Lower bounds for graphs with \texorpdfstring{$\delta(G)\ge3$}{minimum degree at least three}}\label{sec:mainproof}

We prove Theorems~\ref{thm:main} and~\ref{thm:cutaverage} using bounds for sets containing or avoiding a root, together with counting identities at cut edges and cut vertices.

\subsection{Lower bounds for \texorpdfstring{$N(H,r)$}{N(H,r)}}\label{sec:rootcount}

Vince~\cite[Lemma 6.2 and Corollary 6.3]{VinceLower2022} gave linear lower bounds for trees
and for graphs whose root lies in a 2-connected subgraph, with some exceptional graphs.
The stronger linear bound below is also used in Subsection~\ref{sec:cutquantitative} to control
the ratio $F(G)/N(G)$.

We first consider graphs containing a cycle.

\begin{lemma}\label{lem:cyclicroot}
Let $(H,r)$ be a connected rooted graph of order $h$ containing a cycle. Then $N(H,r)\ge h+1$.
\end{lemma}
\begin{proof}
Starting from $r$ in a spanning tree $T$ of $H$, add one adjacent vertex at a time
to obtain $h$ connected sets of orders $1,\ldots,h$.
If any step allows two choices, the other choice gives an additional set of the same order implying that $N(H,r)\ge h+1$. Otherwise, $T$ is a path $r=v_1,\ldots,v_h$, and the sets obtained are $\{v_1,\ldots,v_i\}$ for $1\le i\le h$.
Since $H$ contains a cycle, it has an edge $v_iv_j\notin E(T)$ with $j\ge i+2$.
Then $\{v_1,\ldots,v_i,v_j\}$ is connected and contains $r$,
but differs from every set already counted. 
\end{proof}

\begin{lemma}\label{lem:biconnectedrootcount}
If $B$ is a 2-connected simple graph of order $b$, then, for every $r\in V(B)$,
$N(B,r)\ge\frac{b(b-1)}2+1$.
\end{lemma}
\begin{proof}
Choose a neighbor $u$ of $r$.
By the $st$-numbering theorem (see Even and Tarjan~\cite{EvenTarjan1976}),
$B$ has an ordering $v_1=r,\ldots,v_b=u$ such that every vertex except the first
has an earlier neighbor and every vertex except the last has a later neighbor;
see Figure~\ref{fig:st-numbering}.

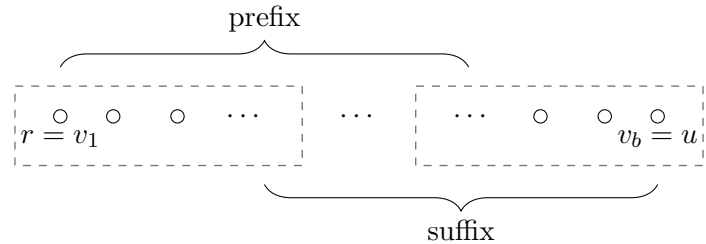
\begin{figure}[htbp]
\centering
\begin{tikzpicture}[font=\small]
  \draw[paper outline] (0,-.65) rectangle (3.8,.4);
  \draw[paper outline] (5.3,-.65) rectangle (9.1,.4);
  \foreach \x in {.6,1.3,2.15,6.95,7.8,8.5}
    \node[paper vertex,fill=white,minimum size=5pt] at (\x,0) {};
  \node at (3.05,0) {$\cdots$};
  \node at (4.55,0) {$\cdots$};
  \node at (6.05,0) {$\cdots$};
  \node[anchor=north,inner sep=1pt] at (.6,-.14) {$r=v_1$};
  \node[anchor=north,inner sep=1pt] at (8.5,-.14) {$v_b=u$};
  \draw[paper brace,decoration={brace,amplitude=10pt}]
    (.6,.65)--(6,.65)
    node[midway,above=9pt] {prefix};
  \draw[paper brace,decoration={brace,mirror,amplitude=10pt}]
    (3.3,-.9)--(8.5,-.9)
    node[midway,below=9pt] {suffix};
\end{tikzpicture}
\caption{The $st$-numbering of $B$}
\label{fig:st-numbering}
\end{figure}

By the $st$-numbering property, every nonempty prefix is connected and contains $r$,
giving $b$ such sets. Every nonempty suffix is also connected.
Since $v_1$ and $v_b$ are adjacent, the union of any nonempty prefix and nonempty suffix
is connected and contains $r$. we only consider disjoint nonempty prefixes and suffixes whose union is not $V(B)$.
For a prefix of order $i$, where $1\le i\le b-2$, the suffix can start at any of
$v_{i+2},\ldots,v_b$, giving $b-i-1$ choices.
These unions are distinct because they omit different intervals, so their number is
$(b-2)+(b-3)+\cdots+2+1=\frac{(b-2)(b-1)}2$.
None is a prefix already counted, since each contains $v_b$ but is not $V(B)$.
Therefore
\begin{equation*}
 N(B,r)\ge b+\frac{(b-2)(b-1)}2=\frac{b(b-1)}2+1.\qedhere
\end{equation*}
\end{proof}

\begin{lemma}\label{lem:rootcount}
Let $(H,r)$ be a connected simple rooted graph of order $h$ in which every nonroot vertex has degree at least $3$.
\begin{enumerate}
 \item If $d_H(r)\ge2$, then
 \begin{equation*}
  N(H,r)\ge4h-8\ge2h.
 \end{equation*}
 Moreover, $N(H,r)=2h$ if and only if $H\cong K_4$.
 \item If $d_H(r)\ge0$, then
 \begin{equation*}
  N(H,r)\ge2h-1.
 \end{equation*}
\end{enumerate}
\end{lemma}
\begin{proof}
We proceed by induction on $h$, proving both assertions simultaneously.
If $h=1$, only the second assertion applies, and it holds.
For $2\le h\le3$, no graph satisfies the degree condition on nonroot vertices.
If $h=4$, clearly $H\cong K_4$ and $N(H,r)=8$.
For $h\ge5$, assume both assertions hold for all smaller orders. We distinguish four cases.

\textbf{Case 1: the root degree is $d:=d_H(r)\ge3$.}
The root together with any subset of its neighborhood gives $2^d$ connected sets containing the root; 

For each nonneighbor $u\notin \Gamma_H(r)\cup\{r\}$, fix a shortest $r$--$u$ path $P_u$, and let $v_u$ be its first
vertex after the root; see Figure~\ref{fig:root-neighbor-path}.
Taking the union of $V(P_u)$ with any subset of $\Gamma_H(r)\setminus\{v_u\}$
gives $2^{d-1}$ connected sets containing the root.

\begin{figure}[htbp]
\centering
\begin{tikzpicture}[font=\small]
  \coordinate (root) at (3.6,0);
  \coordinate (first) at (1.1,.85);
  \coordinate (outside) at (-3.6,.85);
  \draw[paper region] (0,0) ellipse (2.2cm and 1.45cm);
  \node at (-.6,0) {$\Gamma_H(r)$};
  \draw[paper edge] (root)--(1.1,0);
  \draw[paper edge] (root)--(1.1,-.85);
  \draw[paper edge,line width=.9pt] (root)--(first);
  \draw[paper edge,dashed,line width=.9pt] (first)
    .. controls (1.15,3) and (-1.5,3) .. (outside)
    node[pos=.35,above=4pt] {$P_u$};
  \node[paper root,label={right:$r$}] at (root) {};
  \node[paper vertex,minimum size=5pt,label={above left:$v_u$}] at (first) {};
  \node[paper vertex,fill=white,minimum size=5pt,label={below left:$u$}] at (outside) {};
  \foreach \y in {0,-.85}
    \node[paper vertex,fill=white,minimum size=5pt] at (1.1,\y) {};
\end{tikzpicture}
\caption{A shortest $r$--$u$ path $P_u$ through $v_u\in\Gamma_H(r)$.}
\label{fig:root-neighbor-path}
\end{figure}
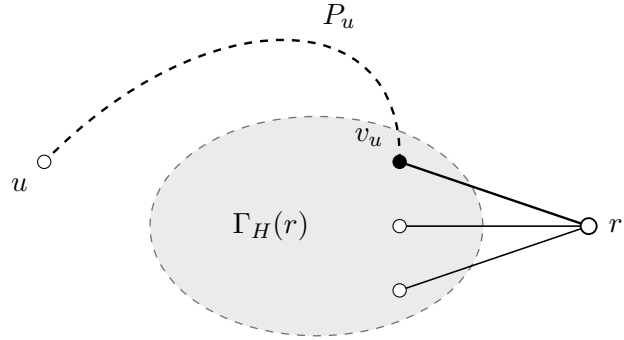

By the choice of $P_u$, we have $V(P_u)\cap\Gamma_H(r)=\{v_u\}$, so different choices
of the subset of neighbors give different sets for each fixed $u$.
Since $P_u$ is shortest, $u$ is the unique vertex farthest from $r$ in each resulting set.
Thus the $h-d-1$ choices of $u$ give $2^{d-1}(h-d-1)$ distinct connected sets containing $r$.
Each contains a vertex outside $\Gamma_H(r)\cup\{r\}$, so none is among the first $2^d$ sets. Hence
\begin{equation*}
 N(H,r)\ge2^d+2^{d-1}(h-d-1)\ge4(1+(d-2))+4(h-d-1)=4h-8.
\end{equation*}
The inequality holds with equality for $d=3$ and is strict for $d>3$.

\textbf{Case 2: the root has degree two and $r$ is a cut vertex.}
The graph $H-r$ has exactly two components, of orders $m_1,m_2$, so $h=m_1+m_2+1$.
Adding the root back to either side gives a graph in which the root has degree one and
every other vertex has degree at least three. Applying the induction hypothesis for \textnormal{(ii)}
to each side and choosing connected sets containing the root independently gives
$N(H,r)\ge(2m_1+1)(2m_2+1)$.
The degree condition ensures that $m_i\ge3$, so
\[
 (2m_1+1)(2m_2+1)-(4h-8)=(2m_1-1)(2m_2-1)+4>0.
\]

\textbf{Case 3: the root has degree two and $r$ is not a cut vertex.}
The two neighbors of the root are connected in $H-r$, so $r$ lies in a 2-connected block $B$.
Write $b=v(B)\ge3$ and $m=h-b$. Both neighbors of $r$ lie in $B$, so $r$ has no neighbor outside $B$.
Every component of $H-V(B)$ has a neighbor in $B$ because $H$ is connected.
It cannot have two distinct neighbors in $B$, since a path between them through that component
would extend $B$ to a larger 2-connected subgraph, contradicting the maximality of the block.
Label the vertices of $B$ other than $r$ as $r_1,\ldots,r_{b-1}$.
For $1\le i\le b-1$, let $O_i$ be the union of the vertex sets of all components of $H-V(B)$
whose unique neighbor in $B$ is $r_i$.
These sets are pairwise disjoint and their union is $V(H)\setminus V(B)$.
Set $m_i=|O_i|$, $H_i=H[O_i\cup\{r_i\}]$, and $q_i=N(H_i,r_i)$.
We regard $H_i$ as rooted at $r_i$; Figure~\ref{fig:block-parts} shows this decomposition.
No edges join distinct sets $O_i$, and every vertex of $O_i$ keeps all its neighbors in $H_i$.
Since $\ord{H_i}=m_i+1<h$, the induction hypothesis for the second assertion gives
$q_i\ge2m_i+1$, including when $m_i=0$. Then we have

\begin{figure}[htbp]
\centering
\begin{tikzpicture}[font=\small]
  \draw[paper region] (0,0) ellipse (2cm and 1.1cm);
  \draw[paper region] (-1.6,.66)
    .. controls (-2.0,2.25) and (-4.15,2.35) .. (-4.3,1.15)
    .. controls (-4.4,.1) and (-2.5,.25) .. cycle;
  \draw[paper region] (1.6,.66)
    .. controls (2.0,2.25) and (4.15,2.35) .. (4.3,1.15)
    .. controls (4.4,.1) and (2.5,.25) .. cycle;
  \node at (0,0) {$B$};
  \node at (-3,1.1) {$O_1$};
  \node at (3,1.1) {$O_{b-1}$};
  \node at (-3,2.25) {$H_1$};
  \node at (3,2.25) {$H_{b-1}$};
  \node at (0,1.8) {$\cdots$};
  \node[paper root,label=below:$r$] at (0,-1.1) {};
  \node[paper root,label=below left:$r_1$] at (-1.6,.66) {};
  \node[paper root,label=below right:$r_{b-1}$] at (1.6,.66) {};
\end{tikzpicture}
\caption{The decomposition of $(H,r)$ when $r$ is not a cut vertex}
\label{fig:block-parts}
\end{figure}
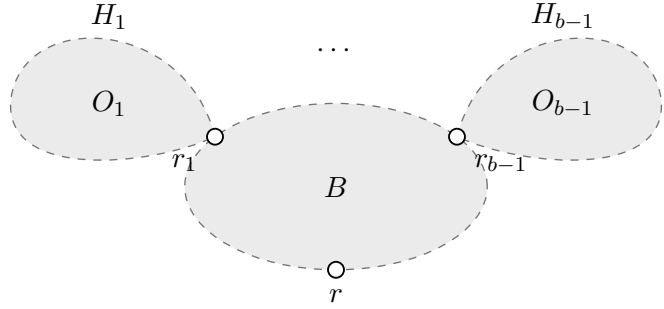

\begin{equation*}
\begin{aligned}
 N(H,r)
 &\ge\sum_{\substack{U\in\CC(B)\\r\in U}}\prod_{i:\,r_i\in U}q_i
 \ge\sum_{\substack{U\in\CC(B)\\r\in U}}
       \left(1+\sum_{i:\,r_i\in U}(q_i-1)\right)\\
 &=N(B,r)+\sum_{i=1}^{b-1}(q_i-1)
       \left|\{U\in\CC(B):\{r,r_i\}\subseteq U\}\right|\\
 &=N(B,r)+\sum_{i=1}^{b-1}(q_i-1)N(B,\{r,r_i\}).
\end{aligned}
\end{equation*}
For each $i$, both the entire vertex set of $B$ and each set obtained by deleting a vertex
outside $\{r,r_i\}$ are connected. Thus $N(B,\{r,r_i\})\ge b-1$.
Combining this with Lemma~\ref{lem:biconnectedrootcount} and $q_i\ge2m_i+1$ gives
\begin{equation*}
 N(H,r)-(4h-8)\ge\frac{b(b-1)}2+1+2(b-1)m-(4h-8)=(b-3)\left(\frac{b-6}{2}+2m\right).
\end{equation*}
The last expression is nonnegative when $b=3$ or $b\ge6$, and also when $b\in\{4,5\}$ and $m\ge1$.
If $m=0,b=4$, the degree condition on nonroot vertices forces the root to have degree three,
a contradiction. If $m=0,b=5$, call the neighbors of the root $u_1,u_2$
and its nonneighbors $v_1,v_2$. Each of $v_1,v_2$ is adjacent to the other three nonroot vertices.
Thus a set containing the root, other than the root alone, is connected if and only if
it contains $u_1$ or $u_2$. Hence $N(H,r)=1+(2^4-2^2)=13>12$.
Cases 1--3 prove the bound in the first assertion.
For $h\ge5$, we have $4h-8>2h$, so $N(H,r)=2h$ holds only when $H=K_4$.
The second assertion also follows in these cases, since $N(H,r)\ge2h>2h-1$.

\textbf{Case 4: the root has degree one.}
Only the second assertion applies. Let $r_1$ be the unique neighbor of $r$. Then $H-r$ is connected,
the degree of $r_1$ in it is at least two, and every other vertex has degree at least three.
The induction hypothesis for the first assertion at a smaller order gives
$N(H,r)=1+N(H-r,r_1)\ge1+2(h-1)=2h-1$.
This completes the induction.
\end{proof}

\subsection{Bounds involving \texorpdfstring{$q(G,B)$}{q(G,B)} and \texorpdfstring{$E(G,B)$}{E(G,B)}}\label{sec:compensation}\label{sec:boundary}

We now consider a more general concept of connected vertex sets introduced by Haslegrave~\cite{HaslegravePath2022}.
For a graph $G$ and $B\subseteq V(G)$, motivated by their work on $(S,H)$-connected sets, we define
\begin{equation*}
 \QQ(G,B)=\{\varnothing\}\cup
 \{U\ne\varnothing:\text{every component of }G[U]\text{ meets }B\}.
\end{equation*}
Write $q(G,B)=|\QQ(G,B)|$ and
$E(G,B)=\sum_{U\in\QQ(G,B)}(2|U|-\ord{G})$.
Note that these quantities are defined whether or not $B$ meets every component of $G$.
Throughout the rest of this subsection, we assume that $B$ meets every component of $G$.
When applying a result to a new graph and vertex set in a proof, we must still check this condition.

Haslegrave~\cite[Claim 2.1]{HaslegravePath2022} gives the following inequality.

\begin{lemma}\label{lem:boundaryhalf}
$E(G,B)\ge0$.
\end{lemma}

The next lemma relates this family to connected sets containing a fixed root.

\begin{lemma}\label{lem:rootdeletion}
Let $(H,r)$ be a connected rooted graph. Then
\begin{equation*}
 q(H-r,\Gamma_H(r))=N(H,r),\qquad E(H-r,\Gamma_H(r))=F(H,r).
\end{equation*}
\end{lemma}
\begin{proof}
Consider the map
\[
 \begin{aligned}
 \varphi\colon\{U\in\CC(H):r\in U\}&\longrightarrow\QQ(H-r,\Gamma_H(r)),\\
 U&\longmapsto U\setminus\{r\}.
 \end{aligned}
\]
For any connected set $U$ containing $r$, every component of $H[U\setminus\{r\}]=(H-r)[U\setminus\{r\}]$ contains a neighbor of $r$, so $\varphi$ is well-defined.
Conversely, if $U\in\QQ(H-r,\Gamma_H(r))$, adding $r$ connects all components of $(H-r)[U]$, while $U=\varnothing$ gives the singleton $\{r\}$, thus $\varphi$ is surjective. For the injection, it's trivial to check. 
Hence $\varphi$ is a bijection, and the first equality follows.
For the second equality,
\[
 \begin{aligned}
 E(H-r,\Gamma_H(r))
 &=\sum_{U\in\QQ(H-r,\Gamma_H(r))}(2|U|-\ord{H-r})\\
 &=\sum_{\substack{U\in\CC(H)\\r\in U}}
   \bigl(2(|U|-1)-(\ord{H}-1)\bigr)\\
 &=2S(H,r)-(\ord{H}+1)N(H,r)=F(H,r).
 \end{aligned}
\]
\end{proof}

Given a graph $G$ and $B\subseteq V(G)$, choose any vertex $r\in B$.
We will introduce several related graphs and their associated quantities, using Figures~\ref{fig:root-deletion}--\ref{fig:root-deletion-H} to make the definitions easier to follow.

\begingroup
\setlength{\intextsep}{8pt}
\setlength{\abovecaptionskip}{6pt}
\begin{figure}[H]
\centering
\begin{tikzpicture}[font=\small]
  \draw[paper component frame]
    (-4.55,-1.5) rectangle (7.95,3.45);
  \node[anchor=west,fill=white,inner sep=2pt] at (-4.3,3.45) {$C_1$};
  \coordinate (root) at (0,2.9);

  \pic (near-only) at (-3.6,0) {paper root branch={1/0}};
  \pic (both) at (-1.2,0) {paper root branch={1/1}};
  \pic (far-only) at (1.2,0) {paper root branch={0/1}};
  \pic[paper near-left y=.65,paper near-right y=1.4] (neither) at (3.6,0) {paper root branch={0/0}};
  \foreach \branchname in {near-only,both,far-only} {
    \draw[paper edge] (root)--(\branchname-near-left);
    \draw[paper edge] (root)--(\branchname-near-right);
  }
  \draw[paper edge] (root)--(neither-near-left);
  \draw[paper edge] (root)--(neither-near-right);
  \node[paper root,fill=papermediumgray,label=above:$r$] at (root) {};
  \foreach \x in {-2.4,0,2.4,4.75} {
    \node at (\x,1.05) {$\cdots$};
    \node at (\x,-.65) {$\cdots$};
  }
  \node[anchor=west,font=\footnotesize,inner sep=0pt]
    at (5.25,1.05) {$\Gamma_G(r)$};
  \node[anchor=west,align=center,font=\footnotesize,inner sep=0pt]
    at (5.25,-.65) {$V(C_1)\setminus$\\[-1pt]$(\Gamma_G(r)\cup\{r\})$};

  \pic {paper lower components};
  \pic {paper B legend={B}};
\end{tikzpicture}
\caption{The structure of $G$. Let $C_1$ be the component of $G$ containing $r$, and let $C_2,C_3,\ldots$ be the other components. The four types of components of $C_1-r$, from left to right, meet $B$ only among neighbors of $r$, among both neighbors and nonneighbors, only among nonneighbors, or not at all. Write $q=q(G,B), E=E(G,B).$}
\label{fig:root-deletion}
\end{figure}
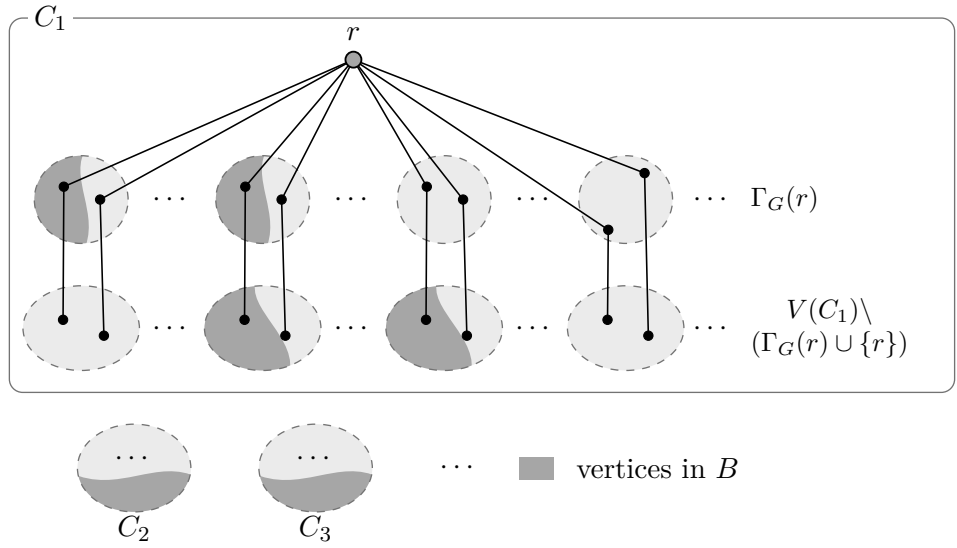

\begin{figure}[H]
\centering
\begin{tikzpicture}[font=\small]
  \draw[paper component frame] (-4.55,-1.5) rectangle (5.25,1.9);
  \pic (prime-near-only) at (-3.6,0) {paper root branch={2/0}};
  \pic (prime-both) at (-1.2,0) {paper root branch={2/1}};
  \pic (prime-far-only) at (1.2,0) {paper root branch={2/1}};
  \pic[paper near-left y=.65,paper near-right y=1.4] (prime-neither) at (3.6,0) {paper root branch={2/0}};
  \foreach \x in {-2.4,0,2.4,4.75} {
    \node at (\x,1.05) {$\cdots$};
    \node at (\x,-.65) {$\cdots$};
  }
  \pic {paper lower components};
  \pic {paper B legend={B'}};
\end{tikzpicture}
\caption{The structure of $G'$ where $G'=G-r$, $B_0=B\setminus\{r\}$, and $B'=B_0\cup\Gamma_G(r)$.
Write $q'=q(G',B')$ and $E'=E(G',B')$.}
\label{fig:root-deletion-prime}
\end{figure}
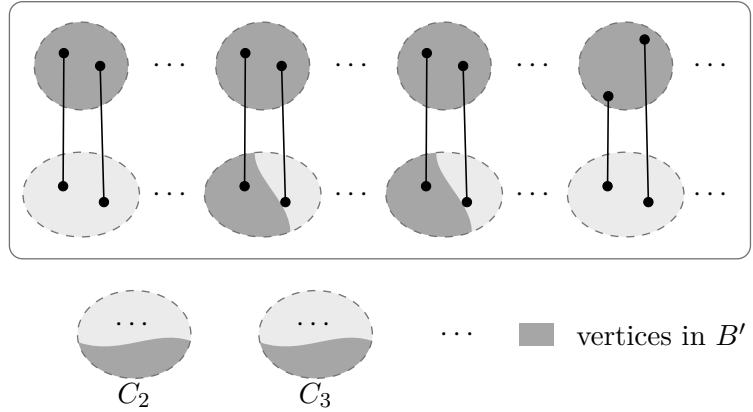

\begin{figure}[H]
\centering
\begin{tikzpicture}[font=\small]
  \draw[paper component frame] (-4.55,-1.5) rectangle (2.9,1.9);
  \pic (zero-near-only) at (-3.6,0) {paper root branch={1/0}};
  \pic (zero-both) at (-1.2,0) {paper root branch={1/1}};
  \pic (zero-far-only) at (1.2,0) {paper root branch={0/1}};
  \foreach \x in {-2.4,0,2.4} {
    \node at (\x,1.05) {$\cdots$};
    \node at (\x,-.65) {$\cdots$};
  }
  \pic {paper lower components};
  \pic {paper B legend={B_0}};
\end{tikzpicture}
\caption{The structure of $G_*$ with $B_0$, where $G_*$ is the union of all components of $G'$ that meet $B_0$.
Write $q_0=q(G_*,B_0)$ and $E_0=E(G_*,B_0)$.}
\label{fig:root-deletion-zero}
\end{figure}
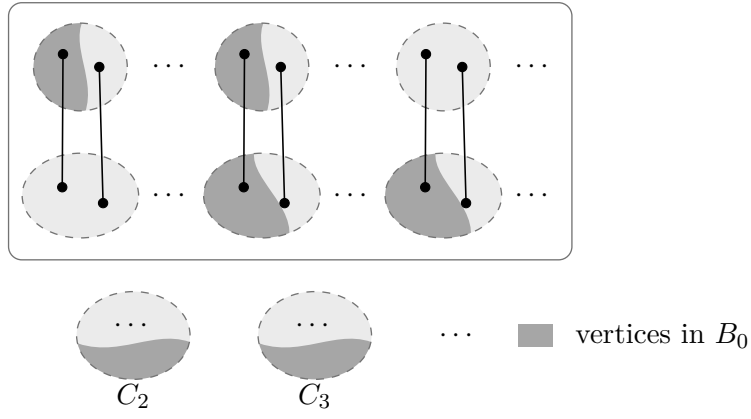

\begin{figure}[H]
\centering
\begin{tikzpicture}[font=\small]
  \draw[paper component frame] (-4.55,-1.5) rectangle (2.9,1.9);
  \pic (star-near-only) at (-3.6,0) {paper root branch={2/0}};
  \pic (star-both) at (-1.2,0) {paper root branch={2/1}};
  \pic (star-far-only) at (1.2,0) {paper root branch={2/1}};
  \foreach \x in {-2.4,0,2.4} {
    \node at (\x,1.05) {$\cdots$};
    \node at (\x,-.65) {$\cdots$};
  }
  \pic {paper lower components};
  \pic {paper B legend={B_*}};
\end{tikzpicture}
\caption{The structure of $G_*$ with $B_*$, where $B_*=B_0\cup(\Gamma_G(r)\cap V(G_*))$.
Write $q_*=q(G_*,B_*)$.}
\label{fig:root-deletion-star}
\end{figure}
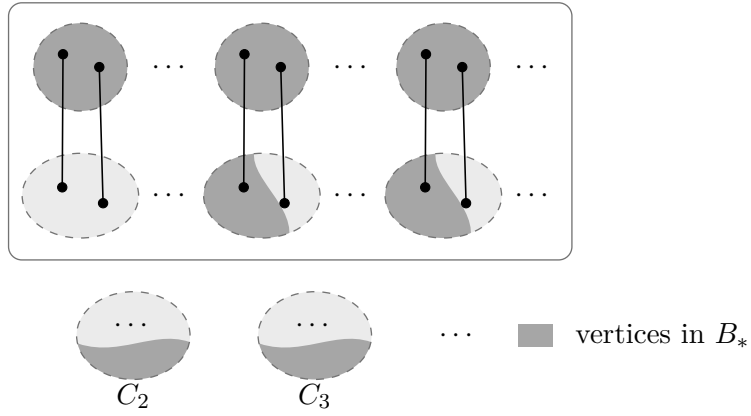

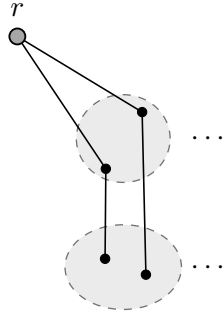
\begin{figure}[H]
\centering
\begin{tikzpicture}[font=\small]
  \coordinate (H-root) at (-1.4,2.4);
  \pic[paper near-left y=.65,paper near-right y=1.4] (H-neither) at (0,0) {paper root branch={0/0}};
  \draw[paper edge] (H-root)--(H-neither-near-left);
  \draw[paper edge] (H-root)--(H-neither-near-right);
  \node[paper root,fill=papermediumgray,label=above:$r$] at (H-root) {};
  \node at (1.15,1.05) {$\cdots$};
  \node at (1.15,-.65) {$\cdots$};
\end{tikzpicture}
\caption{The rooted graph $(H,r)$, where $W=V(G')\setminus V(G_*)$, $\omega=|W|+1$, and $H=G[W\cup\{r\}]$. Write $q_W=N(H,r)$.}
\label{fig:root-deletion-H}
\end{figure}
\endgroup

Since $B$ meets every component of $G$, the definitions illustrated above show that $B'$ meets every component of $G'$ and that both $B_0$ and $B_*$ meet every component of $G_*$ and that $\Gamma_H(r)$ meets every component of $G[W]$.
Thus Lemma~\ref{lem:boundaryhalf} applies to all these pairs.

\begin{lemma}\label{lem:deletionrelations}
With the notation in Figures~\ref{fig:root-deletion}--\ref{fig:root-deletion-H}, we have
\begin{gather}
 q=q'+q_0,\qquad E=E'+q'+E_0-\omega q_0,\label{eq:boundaryrec}\\
 q'=q_Wq_*,\qquad q_*\ge q_0,\qquad N(G)\le N(G')+q'.\label{eq:boundaryproduct}
\end{gather}
\end{lemma}
\begin{proof}
Consider the map
\[
 \begin{aligned}
 \varphi\colon\{U\in\QQ(G,B):r\in U\}&\longrightarrow\QQ(G',B'),\\
 U&\longmapsto U\setminus\{r\}.
 \end{aligned}
\]
As in the proof of Lemma~\ref{lem:rootdeletion}, $\varphi$ is a bijection.
The members of $\QQ(G,B)$ avoiding $r$ form $\QQ(G_*,B_0)$, as shown in Figure~\ref{fig:root-deletion-zero}.
Hence $q=q'+q_0$.
Using $\ord{G'}=\ord{G}-1$ and $\ord{G_*}=\ord{G}-\omega$, we also obtain
\[
 \begin{aligned}
 E
 &=\sum_{U\in\QQ(G',B')}\bigl(2(|U|+1)-\ord{G}\bigr)
   +\sum_{U\in\QQ(G_*,B_0)}\bigl(2|U|-\ord{G}\bigr)\\
 &=E'+q'+E_0-\omega q_0.
 \end{aligned}
\]

Consider the map
\[
 \begin{aligned}
 \psi\colon\QQ(G',B')&\longrightarrow\QQ(G[W],\Gamma_H(r))\times\QQ(G_*,B_*),\\
 U&\longmapsto\bigl(U\cap W,U\cap V(G_*)\bigr).
 \end{aligned}
\]
The diagrams of $G'$, $G[W]$, and $G_*$ in Figures~\ref{fig:root-deletion-prime}--\ref{fig:root-deletion-H} show that every $U\in\QQ(G',B')$ has the natural decomposition given by $\psi$. Thus $\psi$ is well-defined; If $\psi(U_1)=\psi(U_2)$, then $U_1\cap W=U_2\cap W$ and $U_1\cap V(G_*)=U_2\cap V(G_*)$. Taking their union gives $U_1=U_2$, so $\psi$ is injective; For any $(U_1,U_2)\in\QQ(G[W],\Gamma_H(r))\times\QQ(G_*,B_*)$, it's trivial to check $\psi(U_1\cup U_2)=(U_1,U_2)$, so $\psi$ is surjective; hence $\psi$ is a bijection.
Each component of $G[W]$ contains a neighbor of $r$, so $H$ is connected (Figure~\ref{fig:root-deletion-H}).
Lemma~\ref{lem:rootdeletion} gives $q(G[W],\Gamma_H(r))=N(H,r)=q_W$, and hence $q'=q_Wq_*$.

Since $B_0\subseteq B_*$, we have $\QQ(G_*,B_0)\subseteq\QQ(G_*,B_*)$, so $q_*\ge q_0$. Since every connected set containing $r$ belongs to $\QQ(G,B)$ and $\varphi$ is bijective, we have $N(G,r)\le q'$. Thus $N(G)=N(G')+N(G,r)\le N(G')+q'$.
\end{proof}

The next lemma strengthens Haslegrave's bound using degree conditions of vertices outside $B$.

\begin{lemma}\label{lem:compbounds}
Suppose that every $u\in V(G)\setminus B$ has degree at least $3$ in $G$.
\begin{enumerate}
 \item We have
 \begin{equation*}
  q(G,B)+2E(G,B)\ge N(G)+1.
 \end{equation*}
 \item If $|B|\ge2$, then
 \begin{equation*}
  q(G,B)+E(G,B)\ge N(G)+1.
 \end{equation*}
\end{enumerate}
\end{lemma}
\begin{proof}
We proceed by induction on $\ord{G}$, proving both assertions simultaneously.
If $G$ has no vertices, (i) holds with equality and (ii) does not apply.
For nonempty $G$, we have $B\ne\varnothing$ since $B$ meets every component of $G$.
Assume both assertions for all graphs of smaller order and choose $r\in B$.
Every vertex of $W$ lies outside $B$ and has the same degree in $H$ as in $G$ (Figure~\ref{fig:root-deletion-H}).
Lemma~\ref{lem:rootcount}\textnormal{(ii)} gives
$q_W\ge2\omega-1$, so~\eqref{eq:boundaryproduct} gives $q'\ge(2\omega-1)q_0$.

Every vertex of $G'$ outside $B'$ lies outside $B$ and has the same degree in $G'$ as in $G$ (Figure~\ref{fig:root-deletion-prime}).
The induction hypothesis for (i) gives $q'+2E'\ge N(G')+1$.
Using~\eqref{eq:boundaryrec} and~\eqref{eq:boundaryproduct}, we obtain
\begin{align*}
 q+2E-N(G)
 &=(q'+q_0)+2(E'+q'+E_0-\omega q_0)-N(G)\\
 &\ge (q'+q_0)+2(E'+q'+E_0-\omega q_0)-(N(G')+q')\\
 &\ge(q'+2E'-N(G'))+q'-(2\omega-1)q_0+2E_0\\
 &\ge1+q'-(2\omega-1)q_0+2E_0\ge1,
\end{align*}
where the last inequality uses $E_0\ge0$ and the bound $q'\ge(2\omega-1)q_0$.

For (ii), assume $|B|\ge2$.
Since $B_0=B\setminus\{r\}\ne\varnothing$, we have $q_0\ge2$.

\textbf{Case 1: $W=\varnothing$.}
Here $\omega=1$ and $G_*=G'$.
If $|B'|\ge2$, apply the induction hypothesis for (ii) to $(G',B')$.
Using~\eqref{eq:boundaryrec}, \eqref{eq:boundaryproduct}, and $E_0\ge0$, we obtain
$ q+E-N(G)\ge q'+E'-N(G')+E_0\ge1$. If $|B'|=1$, then $B'=B_0$.
Since $G'=G_*$, we have $q'=q_0$ and $E'=E_0$, so
$q=2q_0$ and $E=2E_0$.
The induction hypothesis for (i) on $(G_*,B_0)$ combined with ~\eqref{eq:boundaryproduct} gives
$q+E-N(G)\ge q_0+2E_0-N(G')\ge1$.

\textbf{Case 2: $W\ne\varnothing$.}
Here $\omega\ge2$.
Put $E_W=F(H,r)$ and $E_*=E(G_*,B_*)$.
By Lemma~\ref{lem:rootdeletion}, $q(G[W],\Gamma_H(r))=q_W$ and $E(G[W],\Gamma_H(r))=E_W$.
We have
\begin{equation}\label{eq:strongingredients}
 q_*\ge q_0\ge2,\qquad q_W\ge2\omega-1\ge3.
\end{equation}
Every vertex of $G_*$ outside $B_*$ has the same degree in $G_*$ as in $G$ (Figure~\ref{fig:root-deletion-star}).
Likewise, every vertex of $G[W]$ outside $\Gamma_H(r)$ has the same degree in $G[W]$ as in $G$ (Figure~\ref{fig:root-deletion-H}).
Both $G[W]$ and $G_*$ have smaller order than $G$, so the induction hypothesis for (i) gives
\begin{equation}\label{eq:componentweak}
 N(G[W])\le q_W+2E_W-1,\qquad N(G_*)\le q_*+2E_*-1.
\end{equation}

Under the bijection $\psi$ in the proof of Lemma~\ref{lem:deletionrelations}, choosing $U\in\QQ(G',B')$ is equivalent to independently choosing
$U_1\in\QQ(G[W],\Gamma_H(r))$ and $U_2\in\QQ(G_*,B_*)$ and taking their disjoint union.
Since $|U|=|U_1|+|U_2|$ and $\ord{G'}=|W|+\ord{G_*}$, we have
\[
\begin{aligned}
 E'
 &=\sum_{U\in\mathcal{Q}(G',B')}(2|U|-v(G'))\\
 &=q_*\sum_{U_1}(2|U_1|-|W|)+q_W\sum_{U_2}(2|U_2|-\ord{G_*}).
\end{aligned}
\]
Here each $U_1$ occurs in $q_*$ pairs and each $U_2$ in $q_W$ pairs, and the two sums in the last line are $E_W$ and $E_*$, respectively.
Also, no edges join $W$ to $V(G_*)$, so every connected set of $G'$ lies entirely in one of these two parts. Thus
\begin{equation}\label{eq:momentproduct}
 E'=E_Wq_*+q_WE_*,\qquad N(G')=N(G[W])+N(G_*).
\end{equation}
Combining these identities and bounds, we obtain
\begin{align*}
 q+E-N(G)
 &\ge (q_*-2)E_W+(q_W-2)E_*+E_0+q_W(q_*-1)-q_*-(\omega-1)q_0+2\\
 &\ge q_W(q_*-1)-q_*-(\omega-1)q_0+2\\
 &\ge (q_W-\omega)q_*-q_W+2\\
 &\ge 2(q_W-\omega)-q_W+2\\
 &= q_W-2\omega+2\ge1.
\end{align*}
The first inequality follows from~\eqref{eq:boundaryrec} and~\eqref{eq:momentproduct}, using~\eqref{eq:boundaryproduct} and~\eqref{eq:componentweak} to bound $N(G)$.
For the second, Lemma~\ref{lem:boundaryhalf} gives $E_W,E_*,E_0\ge0$, and~\eqref{eq:strongingredients} gives $q_*\ge2$ and $q_W\ge3$.
The third uses $q_0\le q_*$ with $\omega\ge2$; the fourth uses $q_*\ge2$ with $q_W-\omega\ge\omega-1\ge1$; and the last uses $q_W\ge2\omega-1$, all from~\eqref{eq:strongingredients}. This proves (ii).
\end{proof}

\begin{corollary}\label{cor:rhononpositive}
Suppose that $(H,r)$ is connected, that $d_H(r)\ge2$, and that every
$u\in V(H)\setminus(\Gamma_H(r)\cup\{r\})$ has degree at least $3$. Then
$N(H,r)+F(H,r)\ge\overline N(H,r)+1$.
\end{corollary}
\begin{proof}
The set $\Gamma_H(r)$ meets every component of $H-r$ and contains at least two vertices.
Every vertex of $H-r$ outside $\Gamma_H(r)$ has the same degree in $H-r$ as in $H$.
The claim follows from Lemmas~\ref{lem:compbounds}\textnormal{(ii)} and~\ref{lem:rootdeletion}.
\end{proof}

We next determine when $E(G,B)=0$ under a weaker degree condition outside $B$.
This will give the equality cases in the main theorem.

\begin{lemma}\label{lem:stricthalf}
Suppose that every $u\in V(G)\setminus B$ has degree at least $2$. Then
$E(G,B)=0\Longleftrightarrow B=V(G)$.
\end{lemma}
\begin{proof}
If $B=V(G)$, then $\QQ(G,B)=2^{V(G)}$. Then we have
\[
  E(G,B)=\sum_{U\in\QQ(G,B)}\big(2|U|-v(G)\big)=2\sum_{U\subseteq V(G)}|U|-v(G)2^{v(G)}=0.
\]

This also holds for the graph with no vertices.

Suppose that $B\ne V(G)$. Since $B$ meets every component, there is an edge $ru$ with $r\in B$ and $u\notin B$.
If $W=\varnothing$, then $G_*=G'$ and $\omega=1$.
Enlarging $B_0$ to $B'$ adds at least the singleton $\{u\}$ to the family,
so $q'\ge q_0+1$.
By Lemma~\ref{lem:boundaryhalf} and~\eqref{eq:boundaryrec},
$E=E'+E_0+q'-q_0\ge q'-q_0\ge1$; If $W\ne\varnothing$, the connected graph $H$ has root degree at least $1$.
Every nonroot vertex of $H$ lies outside $B$ and has the same degree in $H$ as in $G$ (Figure~\ref{fig:root-deletion-H}), so its degree is at least $2$. Thus the sum of degrees in $H$ is at least $1+2(\omega-1)>2(\omega-1)$, which implies that $H$ contains a cycle.
Lemma~\ref{lem:cyclicroot} gives $q_W\ge\omega+1$.
By~\eqref{eq:boundaryrec} and~\eqref{eq:boundaryproduct},
$E=E'+E_0+q_Wq_*-\omega q_0\ge q_Wq_*-\omega q_0\ge (\omega+1)q_0-\omega q_0= q_0\ge1$.
Both cases give a strict inequality, so equality does not hold.
\end{proof}

\begin{corollary}\label{cor:universalequality}
Let $(H,r)$ be a connected rooted graph in which every nonroot vertex not adjacent to $r$ has degree at least $2$ in $H$. Then
$F(H,r)=0$ if and only if $r$ is adjacent to every other vertex of $H$.
\end{corollary}
\begin{proof}
Apply Lemma~\ref{lem:stricthalf} to $(H-r,\Gamma_H(r))$.
Every vertex of $H-r$ outside $\Gamma_H(r)$ is not adjacent to $r$ and hence has the same degree in $H-r$ as in $H$.
The result follows from Lemma~\ref{lem:rootdeletion}.
\end{proof}

\subsection{Proof of Vince's conjecture}\label{sec:repair}

Using the preceding inequalities involving $N(G)$, $F(G)$, $E(G,B)$, and related quantities, we now prove Vince's conjecture.
Take two disjoint connected simple rooted graphs $(H_i,r_i)$ in which each root has degree at least $2$
and every nonroot vertex has degree at least $3$.
Set $n_i=\ord{H_i}$, $S_i=S(H_i,r_i)$, $w_i=S(H_i-r_i)$, $F_i=F(H_i,r_i)$, and $\rho_i=\overline N(H_i,r_i)+1-N(H_i,r_i)-F_i$.
The preceding results give
\begin{equation}\label{eq:signs}
 n_i\ge4,\qquad N(H_i,r_i)\ge2n_i,\qquad F_i\ge0,\qquad\rho_i\le0.
\end{equation}
Set $R=N(H_1,r_1)N(H_2,r_2)-n_2N(H_1,r_1)-n_1N(H_2,r_2)$.
Let $G_0$ be obtained by identifying the roots, and let $G_1$ be obtained by adding the edge $r_1r_2$.
For $\lambda\in\{0,1\}$, their orders are $\ord{G_\lambda}=n_1+n_2-1+\lambda$.

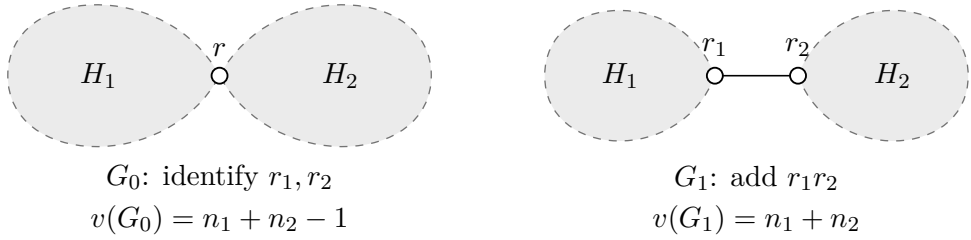
\begin{figure}[htbp]
\centering
\begin{tikzpicture}[font=\small]
  \begin{scope}
    \draw[paper region] (0,0)
      .. controls (-.7,1.3) and (-2.8,1.2) .. (-2.8,0)
      .. controls (-2.8,-1.2) and (-.7,-1.3) .. cycle;
    \draw[paper region] (0,0)
      .. controls (.7,1.3) and (2.8,1.2) .. (2.8,0)
      .. controls (2.8,-1.2) and (.7,-1.3) .. cycle;
    \node at (-1.6,0) {$H_1$};
    \node at (1.6,0) {$H_2$};
    \node[paper root,label=above:$r$] at (0,0) {};
    \node at (0,-1.35) {$G_0$: identify $r_1,r_2$};
    \node at (0,-1.9) {$\ord{G_0}=n_1+n_2-1$};
  \end{scope}
  \begin{scope}[xshift=7.1cm]
    \draw[paper region] (-.55,0)
      .. controls (-1.0,1.2) and (-2.8,1.1) .. (-2.8,0)
      .. controls (-2.8,-1.1) and (-1.0,-1.2) .. cycle;
    \draw[paper region] (.55,0)
      .. controls (1.0,1.2) and (2.8,1.1) .. (2.8,0)
      .. controls (2.8,-1.1) and (1.0,-1.2) .. cycle;
    \draw[paper edge] (-.55,0)--(.55,0);
    \node at (-1.8,0) {$H_1$};
    \node at (1.8,0) {$H_2$};
    \node[paper root,label=above:$r_1$] at (-.55,0) {};
    \node[paper root,label=above:$r_2$] at (.55,0) {};
    \node at (0,-1.35) {$G_1$: add $r_1r_2$};
    \node at (0,-1.9) {$\ord{G_1}=n_1+n_2$};
  \end{scope}
\end{tikzpicture}
\end{figure}

For each $i$, obtain $H_i^+$ from $H_i$ by adding a new vertex $r_i'$ adjacent precisely to the vertices in $\Gamma_{H_i}(r_i)\cup\{r_i\}$, as illustrated below.
Both $r_i$ and $r_i'$ have degree at least $3$, and the degrees of all other vertices do not decrease. Hence
$\delta(H_i^+)\ge3$ and $\ord{H_i^+}=n_i+1$.

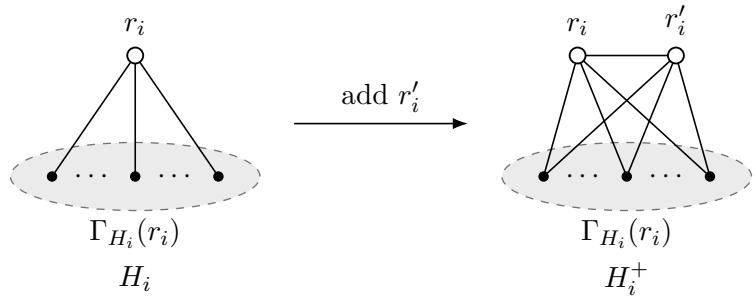
\begin{figure}[htbp]
\centering
\begin{tikzpicture}[font=\small,>=Latex]
  \begin{scope}
    \draw[paper region] (0,-.3) ellipse (1.65cm and .45cm);
    \foreach \x in {-1.1,0,1.1} {
      \draw[paper edge] (0,1.3)--(\x,-.3);
      \node[paper vertex] at (\x,-.3) {};
    }
    \node at (-.55,-.3) {$\cdots$};
    \node at (.55,-.3) {$\cdots$};
    \node[paper root,label=above:$r_i$] at (0,1.3) {};
    \node at (0,-1.05) {$\Gamma_{H_i}(r_i)$};
    \node at (0,-1.65) {$H_i$};
  \end{scope}
  \draw[->,paper edge] (2.1,.4)--(4.4,.4) node[midway,above] {add $r_i'$};
  \begin{scope}[xshift=6.5cm]
    \draw[paper region] (0,-.3) ellipse (1.65cm and .45cm);
    \foreach \x in {-1.1,0,1.1} {
      \draw[paper edge] (-.65,1.3)--(\x,-.3);
      \draw[paper edge] (.65,1.3)--(\x,-.3);
      \node[paper vertex] at (\x,-.3) {};
    }
    \draw[paper edge] (-.65,1.3)--(.65,1.3);
    \node at (-.55,-.3) {$\cdots$};
    \node at (.55,-.3) {$\cdots$};
    \node[paper root,label=above:$r_i$] at (-.65,1.3) {};
    \node[paper root,label=above:$r_i'$] at (.65,1.3) {};
    \node at (0,-1.05) {$\Gamma_{H_i}(r_i)$};
    \node at (0,-1.65) {$H_i^+$};
  \end{scope}
\end{tikzpicture}
\end{figure}

\begin{proposition}\label{prop:twins}
Under the above assumptions, for $\lambda\in\{0,1\}$,
\begin{align}
 F(G_\lambda)-F(H_1^+)-F(H_2^+)={}&\ord{G_\lambda}-3+\lambda+(1+\lambda)R\notag\\
 &+F_1(N(H_2,r_2)-n_2-1)+F_2(N(H_1,r_1)-n_1-1)\notag\\
 &-(n_2-2+\lambda)\rho_1-(n_1-2+\lambda)\rho_2.
 \label{eq:twinexact}
\end{align}
In particular,
\begin{equation}\label{eq:twincomparison}
 F(G_\lambda)\ge F(H_1^+)+F(H_2^+)+\ord{G_\lambda}-3+\lambda.
\end{equation}
\end{proposition}
\begin{proof}
A simple calculation gives the following equalities:
\begin{equation}\label{eq:cutNW}
 \begin{aligned}
 N(G_0)&=\overline N(H_1,r_1)+\overline N(H_2,r_2)+N(H_1,r_1)N(H_2,r_2),\\
 S(G_0)&=w_1+w_2+S_1N(H_2,r_2)+N(H_1,r_1)S_2-N(H_1,r_1)N(H_2,r_2),
 \end{aligned}
\end{equation}
and
\begin{equation}\label{eq:bridgeNW}
 \begin{aligned}
 N(G_1)&=N(H_1)+N(H_2)+N(H_1,r_1)N(H_2,r_2),\\
 S(G_1)&=w_1+w_2+S_1+S_2+S_1N(H_2,r_2)+N(H_1,r_1)S_2.
 \end{aligned}
\end{equation}
For $i=1,2$, we also have
\begin{equation*}
 \begin{aligned}
 N(H_i^+)&=\overline N(H_i,r_i)+3N(H_i,r_i),\\
 S(H_i^+)&=w_i+3S_i+N(H_i,r_i).
 \end{aligned}
\end{equation*}
Using these equalities and $2S_i=(n_i+1)N(H_i,r_i)+F_i$ in the definition of $F$, we obtain
\begin{align*}
 F(G_\lambda)={}&2(w_1+w_2)+(1+\lambda)N(H_1,r_1)N(H_2,r_2)\\
 &+F_1\bigl(N(H_2,r_2)+\lambda\bigr)+F_2\bigl(N(H_1,r_1)+\lambda\bigr)\\
 &-\lambda(n_2-2+\lambda)N(H_1,r_1)-\lambda(n_1-2+\lambda)N(H_2,r_2)\\
 &-\ord{G_\lambda}\bigl(\overline N(H_1,r_1)+\overline N(H_2,r_2)\bigr),\\
 F(H_i^+)={}&2w_i+3F_i+2N(H_i,r_i)-(n_i+1)\overline N(H_i,r_i).
\end{align*}
Substituting $\overline N(H_i,r_i)=N(H_i,r_i)+F_i+\rho_i-1$
into $F(G_\lambda)-F(H_1^+)-F(H_2^+)$ gives~\eqref{eq:twinexact},
using $\ord{G_\lambda}=n_1+n_2-1+\lambda$ and the definition of $R$.

Since $N(H_i,r_i)\ge2n_i$, we have
\[
 R=\tfrac12\bigl(N(H_2,r_2)(N(H_1,r_1)-2n_1)+N(H_1,r_1)(N(H_2,r_2)-2n_2)\bigr)\ge0,
\]
with equality if and only if $N(H_i,r_i)=2n_i$ for $i=1,2$.
Also, $N(H_i,r_i)-n_i-1\ge n_i-1>0$ and $n_i-2+\lambda>0$.
Applying these bounds and~\eqref{eq:signs} to~\eqref{eq:twinexact} gives~\eqref{eq:twincomparison}.
\end{proof}

\begin{proof}[Proof of Theorem~\ref{thm:main}]
We proceed by induction on \(n\), assuming that \(F(H)\ge v(H)\) for every connected simple graph \(H\) of smaller order with minimum degree at least \(3\).
Clearly, $n\ge4$. If $n=4$, then $G\cong K_4$ and $F(G)=4=n$.
If $G$ has no cut vertex, then it is 2-connected, and Theorem~\ref{thm:biconnected} gives $F(G)\ge n$.

Suppose that $G$ has a cut vertex. If it has a cut edge, delete that edge to obtain two rooted sides $H_1,H_2$.
Each root loses only the cut edge and hence has degree at least $2$ on its side; in this case $G=G_1$.
Otherwise, choose a cut vertex $r$, separate one component of $G-r$ from the remaining components,
and add $r$ back to each side. This gives $G=G_0$.
In the latter case, each component contains at least two neighbors of $r$; otherwise, its unique edge to $r$ would be a cut edge.
Thus both decompositions satisfy the root-degree conditions above.
For $i=1,2$, every nonroot vertex of $H_i$ has the same degree in $H_i$ as in $G$.
In either case, $G=G_\lambda$ and $n=n_1+n_2-1+\lambda$ for some $\lambda\in\{0,1\}$.
The graphs $H_i^+$ are connected and simple with minimum degree at least $3$, and
$n-\ord{H_i^+}=n_j-2+\lambda\ge2$ for $j\ne i\implies v(H_i^+)<n$.
The induction hypothesis therefore applies to both.
By~\eqref{eq:twincomparison},
\[
 F(G)\ge(n_1+1)+(n_2+1)+n-3+\lambda=2n.
\]
This proves the general bound and the stronger bound for graphs with a cut vertex.

We now determine the equality cases for the two bounds:

For the first inequality, 
if $G=K_n$, then~\eqref{eq:complete} gives $F(G)=n$.
Conversely, suppose that $F(G)=n$.
The graph $G$ cannot have a cut vertex, otherwise $F(G)\ge2n>n$. Thus $G$ is 2-connected. Choose the edge $uv$ given by Lemma~\ref{lem:descent}. Then $G_{uv}$ is 2-connected, and Lemma~\ref{lem:contraction} gives
\[
 n=F(G_{uv})+\bigl(N(G,\{u,v\})-\overline N(G,\{u,v\})\bigr)+F(G-v,u)+F(G-u,v).
\]
The four summands on the right have lower bounds $n-1,1,0,0$, respectively, so equality must hold in each bound, namely $F(G_{uv})=n-1,N(G,\{u,v\})-\overline N(G,\{u,v\})=1,F(G-v,u)=F(G-u,v)=0$.
Since $\delta(G)\ge3$, the vertex-deleted graphs $G-v$ and $G-u$ have minimum degree at least $2$.
Applying Corollary~\ref{cor:universalequality} to each shows that both $u$ and $v$ are adjacent to every other vertex of $G$. Hence every vertex set containing $u,v$ is connected, and there are $2^{n-2}$ such sets.
Since $N(G,\{u,v\})-\overline N(G,\{u,v\})=1$, we have
$\overline N(G,\{u,v\})=N(G-\{u,v\})=2^{n-2}-1$.
Thus every nonempty vertex set of $G-\{u,v\}$ is connected directly implying that it's complete. Adding back the two vertices $u,v$, each adjacent to every other vertex, shows that $G$ is complete.

For the second inequality, if $G\cong G_{\mathrm b}$ or $G\cong G_{\mathrm c}$, then clearly $\delta(G)=3$ and $G$ has a cut vertex.
After a simple calculation:
\begin{table}[H]
\centering
\label{tab:cut-extremal-counts}
\setlength{\tabcolsep}{12pt}
\begin{tabular}{ccccc}
\toprule
$G$ & $\ord{G}$ & $N(G)$ & $S(G)$ & $F(G)$ \\
\midrule
$G_{\mathrm b}$ & 8 & 94 & 384 & 16 \\
$G_{\mathrm c}$ & 7 & 78 & 280 & 14 \\
\bottomrule
\end{tabular}
\end{table}
Thus $F(G)=2n$ in both cases. Conversely, suppose that $G$ has a cut vertex and $F(G)=2n$, and use the same decomposition $G=G_\lambda$ as above.
Identity~\eqref{eq:twinexact} and $F(H_i^+)\ge n_i+1$ give
$F(G)\ge2n+(1+\lambda)R$.
Hence $R=0$, so $N(H_i,r_i)=2n_i$ for $i=1,2$.
Lemma~\ref{lem:rootcount} then gives $H_1=H_2=K_4$.
Thus $G\cong G_{\mathrm c}$ when $\lambda=0$, and $G\cong G_{\mathrm b}$ when $\lambda=1$.
\end{proof}

\begin{corollary}\label{cor:extremal}
If $G$ is a 2-connected simple graph of order $n$, or a connected simple graph of order $n$ with minimum degree at least $3$, then
\begin{equation*}
 A(G)\ge A(K_n)=\frac{n2^{n-1}}{2^n-1},
\end{equation*}
with equality if and only if $G\cong K_n$.
\end{corollary}
\begin{proof}[Proof of Corollary~\ref{cor:extremal}]
By Theorem~\ref{thm:main} or Theorem~\ref{thm:biconnected}, followed by $N(G)\le2^n-1$, we obtain
\[
 A(G)=\frac n2+\frac{F(G)}{2N(G)}
       \ge\frac n2+\frac{n}{2N(G)} \ge\frac n2+\frac{n}{2(2^n-1)}
       =A(K_n).
\]
The equality holds when $G$ is complete.
\end{proof}

\subsection{Lower bounds for \texorpdfstring{$A(G)$}{A(G)} in graphs with a cut vertex}\label{sec:cutquantitative}

Using~\eqref{eq:twinexact}, we can obtain a stronger lower bound for $F(G)/N(G)$.

\begin{proof}[Proof of Theorem~\ref{thm:cutaverage}]
The proofs of \textnormal{(i)} and \textnormal{(ii)} differ only in the calculations, so we give the details for \textnormal{(i)}. Suppose that $G$ has a cut vertex but no cut edge. Splitting at a cut vertex gives $G=G_0$.
Lemma~\ref{lem:rootcount} gives $N(H_i,r_i)\ge4n_i-8$ and $n_i\ge4$.

Substituting $\overline N(H_i,r_i)=N(H_i,r_i)+F_i+\rho_i-1$ into~\eqref{eq:cutNW}, and using $\rho_i\le0$, gives
\begin{equation}\label{eq:quant-N}
 N(G_0)\le N(H_1,r_1)N(H_2,r_2)+N(H_1,r_1)+N(H_2,r_2)+F_1+F_2-2.
\end{equation}
By Theorem~\ref{thm:main}, $F(H_i^+)\ge n_i+1$.
Together with~\eqref{eq:twinexact} for $\lambda=0$, the signs in~\eqref{eq:signs}, and $N(H_i,r_i)-n_i-1\ge3$, this yields
\begin{equation*}
 F(G_0)\ge R+2(n_1+n_2)-2+3(F_1+F_2).
\end{equation*}
The coefficient of $n_i$ on the right is $2-N(H_j,r_j)<0$ for $j\ne i$.
Thus substituting $n_i\le(N(H_i,r_i)+8)/4$ gives
\begin{equation*}
 F(G_0)\ge\tfrac12N(H_1,r_1)N(H_2,r_2)-\tfrac32(N(H_1,r_1)+N(H_2,r_2))+6+3(F_1+F_2).
\end{equation*}
Combining this estimate with~\eqref{eq:quant-N}, using $F_i\ge0$, and putting $z_i=N(H_i,r_i)-8\ge0$, we obtain
\begin{equation*}
 39F(G_0)-7N(G_0)\ge\tfrac{25}{2}z_1z_2+\tfrac{69}{2}(z_1+z_2)\ge0.
\end{equation*}
Since $A(G)=n/2+F(G)/(2N(G))$, it follows that $A(G)\ge n/2+7/78$. The table above gives $F(G_{\mathrm c})/N(G_{\mathrm c})=7/39$, so equality holds for $G_{\mathrm c}$.
Conversely, if $A(G)=n/2+7/78$, then $z_1=z_2=0$, so $N(H_1,r_1)=N(H_2,r_2)=8$.
Then $4\le n_i\le(N(H_i,r_i)+8)/4=4$, and the degree condition on the nonroot vertices forces $H_i\cong K_4$ for $i=1,2$.
Hence $G\cong G_{\mathrm c}$.
\end{proof}

\section{Lower bounds for graphs with \texorpdfstring{$\delta(G)=k\ge3$}{minimum degree exactly k at least three}}\label{sec:fixeddegree}

Motivated by Vince's conjecture solved above, we now seek to sharpen the bound $A(G)\ge n/2$ for graphs with prescribed minimum degree.
Throughout this section, unless otherwise stated, $G$ is a connected simple graph of order $n$ with $\delta(G)=k\ge3$, and we write $m=n-k-1$.

\subsection{Bounds for \texorpdfstring{$N(G)$ and $F(G)$}{N(G) and F(G)}}

We compare $G$ with $H_{n,k}$ by adding edges while preserving the neighborhood of a vertex of minimum degree.
For any graph $J$ and vertex $v\in V(J)$, let $J^v$ be obtained from $J$ by adding an edge between every pair of nonadjacent vertices in $V(J)\setminus\{v\}$.
Write $\calA(J,v)=\CC(J^v)\setminus\CC(J)$ for the family of connected sets created by adding these edges.
Choose $v\in V(G)$ with $d_G(v)=k$.
The neighborhood of $v$ is unchanged, so $G^v\cong H_{n,k}$.

\begin{figure}[htbp]
\centering
\begin{tikzpicture}[font=\small]
  \draw[paper region] (0,0) ellipse (1.05cm and 1cm);
  \draw[paper region] (3.7,0) ellipse (1.05cm and 1cm);
  \coordinate (comparison-v) at (-3.6,0);
  \coordinate (kk-near-upper) at (-.58,.46);
  \coordinate (kk-near-lower) at (-.58,-.46);
  \coordinate (kk-join-upper) at (.65,.38);
  \coordinate (kk-join-lower) at (.65,-.38);
  \coordinate (km-join-upper) at (3.05,.38);
  \coordinate (km-join-lower) at (3.05,-.38);
  \draw[paper edge] (comparison-v)--(kk-near-upper);
  \draw[paper edge] (comparison-v)--(kk-near-lower);
  \draw[paper edge] (kk-join-upper)--(km-join-upper);
  \draw[paper edge] (kk-join-lower)--(km-join-lower);
  \node[paper root,label={below=4pt:$v$}] at (comparison-v) {};
  \foreach \endpoint in {kk-near-upper,kk-near-lower,kk-join-upper,kk-join-lower,km-join-upper,km-join-lower}
    \node[paper vertex] at (\endpoint) {};
  \node at (0,0) {$K_k$};
  \node at (3.7,0) {$K_m$};
  \node at (0,-1.3) {$\Gamma_{H_{n,k}}(v)$};
  \node at (3.7,-1.3) {$m=n-k-1$};
  \draw[paper brace,decoration={brace,mirror,amplitude=6pt}]
    (-1.15,-1.8)--(4.85,-1.8) node[midway,below=9pt] {$K_{n-1}$};
\end{tikzpicture}
\caption{$H_{n,k}=K_k\vee(K_1\cup K_m)$ for $m\ge0$.}
\label{fig:comparison-graph}
\end{figure}
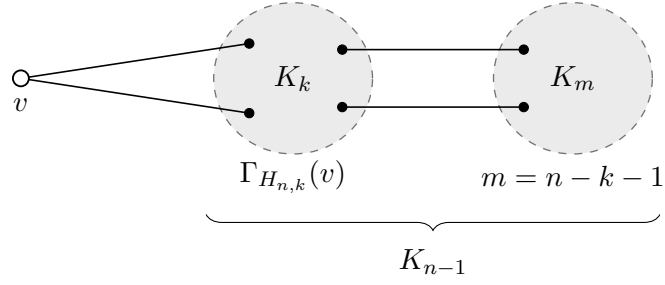

\begin{lemma}\label{lem:candidate-count}
$N(H_{n,k})=2^n-2^m$, $S(H_{n,k})=n2^{n-1}-(m+2)2^{m-1}+1$, and $F(H_{n,k})=2+(k-1)2^m$.
In particular, $A(H_{n,k})=\mu_{n,k}>n/2$.
\end{lemma}
\begin{proof}
The formulas follow by viewing $H_{n,k}$ as a spanning subgraph of $K_n$ and comparing their families of connected sets.
Since $F(H_{n,k})=2+(k-1)2^m>0$, we have $A(H_{n,k})=n/2+F(H_{n,k})/(2N(H_{n,k}))>n/2$.
\end{proof}

The next lemma shows the changes in $N$ and $F$ in terms of $\calA(G,v)$.

\begin{lemma}\label{lem:completion}
$N(H_{n,k})-N(G)=|\calA(G,v)|$ and $F(G)-F(H_{n,k})=\sum_{U\in\calA(G,v)}(n-2|U|)$.
Moreover, $\calA(G,v)=\varnothing$ if and only if $G=G^v$, or equivalently, $G\cong H_{n,k}$.
\end{lemma}
\begin{proof}
Since $\CC(G)\subseteq\CC(G^v)$, we have
$N(H_{n,k})-N(G)=|\calA(G,v)|$ and $S(H_{n,k})-S(G)=\sum_{U\in\calA(G,v)}|U|$.
Substitution into the definition of $F$ gives the second identity.
If an edge is added, its two endpoints form a set in $\calA(G,v)$; hence $\calA(G,v)=\varnothing$ if and only if $G=G^v$.
\end{proof}

\begin{corollary}\label{cor:completion-comparison}
If $F(G)\ge F(H_{n,k})$, then $A(G)\ge\mu_{n,k}$, with equality if and only if $G\cong H_{n,k}$.
\end{corollary}
\begin{proof}
By Lemmas~\ref{lem:candidate-count} and~\ref{lem:completion},
\[
 A(G)=\frac n2+\frac{F(G)}{2N(G)}
 \ge\frac n2+\frac{F(H_{n,k})}{2N(H_{n,k})}=\mu_{n,k}.
\]
The equality holds precisely when $G\cong H_{n,k}$.
\end{proof}

By Lemma~\ref{lem:completion}, the bound $F(G)\ge F(H_{n,k})$ is equivalent to
$2\sum_{U\in\calA(G,v)}|U|\le n|\calA(G,v)|$.
To prove this for joins, we first give a bound for a family of subsets of a finite set.

\begin{lemma}\label{lem:deletion-family}
Let $X$ be an $h$-element set and let $\mathcal F\subseteq2^X$.
If every $U\in\mathcal F$ has at least $|U|-2$ elements $u$ such that $U\setminus\{u\}\in\mathcal F$, then
\begin{equation}\label{eq:deletion-family}
 2\sum_{U\in\mathcal F}|U|\le(h+2)|\mathcal F|.
\end{equation}
\end{lemma}
\begin{proof}
Count pairs $(U',U)\in\mathcal F\times\mathcal F$ with $U'\subset U$ and $|U\setminus U'|=1$.
For each $U$, there are at least $|U|-2$ choices of $U'$; for each $U'$, there are at most $h-|U'|$ choices of $U$.
Thus
\[
 \sum_{U\in\mathcal F}(|U|-2)\le\sum_{U\in\mathcal F}(h-|U|),
\]
which gives~\eqref{eq:deletion-family}, including when $\mathcal F=\varnothing$.
\end{proof}

For a graph $J$, define $\DD(J)=\{\varnothing\ne U\subseteq V(J):J[U]\text{ is disconnected}\}.$
We use $\DD(J)$ and $\calA(J,v)$ for the two sides of a join.

\begin{lemma}\label{lem:two-families}
For any graph $J$ of order $h$ and any $v\in V(J)$, both $\DD(J)$ and $\calA(J,v)$ satisfy~\eqref{eq:deletion-family}.
\end{lemma}
\begin{proof}
For $U\in\DD(J)$, deleting a vertex can make $J[U]$ connected only if $J[U]$ has exactly two components and the deleted vertex forms one of them.
There are at most two such vertices, so Lemma~\ref{lem:deletion-family} applies to $\DD(J)$.

Now suppose that some $U\in\calA(J,v)$ with $|U|\ge3$ has three distinct vertices $u_1,u_2,u_3$ such that $U\setminus\{u_i\}\notin\calA(J,v)$ for $i=1,2,3$.
Since $J[U]$ is disconnected, $J[U\setminus\{u_i\}]$ can be connected only if $J[U]$ has exactly two components, one of which is $\{u_i\}$.
As $|U|\ge3$, this occurs for at most one $i$.
Thus $J^v[U\setminus\{u_i\}]$ must be disconnected for at least two indices $i$.
Since $J^v[U]$ is connected and $J^v-v$ is complete, this forces $v\in U$ and $U\cap\Gamma_J(v)=\{u_i\}$ for two distinct indices, a contradiction.
Therefore at least $|U|-2$ deletions preserve membership in $\calA(J,v)$; for $|U|\le2$, this is automatic.
Lemma~\ref{lem:deletion-family} applies.
\end{proof}

\subsection{Proof of Theorem~\ref{thm:fixeddegree}}

\begin{proof}[Proof of Theorem~\ref{thm:fixeddegree}]
\textbf{(i)} Suppose that $G$ has a cut vertex $r$.
Each component of $G-r$ has at least $k$ vertices, since every neighbor of a vertex in that component lies in the component or is $r$.
Hence $n\ge2k+1$ and $m\ge k\ge3$.
By the definition of $\mu_{n,k}$ and Theorem~\ref{thm:cutaverage},
\[
 \mu_{n,k}-\frac n2=\frac{k-1}{2(2^{k+1}-1)}+\frac1{(2^{k+1}-1)2^m}
 \le \frac{3-1}{2(2^{3+1}-1)}+\frac1{(2^{3+1}-1)2^3}<\frac4{47}\le A(G)-\frac n2.
\]

\smallskip\noindent\textbf{(ii)} If $G\cong H_{n,k}$, then Lemma~\ref{lem:candidate-count} gives $A(G)=\mu_{n,k}$.
By definition, $H_{n,k}$ satisfies condition \textnormal{(ii)}. Now suppose that $G=G[A]\vee G[B]$, where $a=|A|\ge2$ and $b=|B|\ge2$.
Choose $v$ of degree $k$ and, without loss of generality, assume that $v\in A$.
Let $\calA_1=\calA(G[A],v)$, $\calA_2=\DD(G[B])$, and $m_i=|\calA_i|$.
Every vertex set meeting both $A$ and $B$ is already connected in $G$, so $\calA(G,v)$ is disjoint union of $\calA_1$ and $\calA_2.$
By Lemma~\ref{lem:two-families},
\[
 2\sum_{U\in\calA_1}|U|\le(a+2)m_1,\qquad
 2\sum_{U\in\calA_2}|U|\le(b+2)m_2.
\]
Using these bounds in the second identity of Lemma~\ref{lem:completion}, with $n=a+b$, gives
\[
 \begin{aligned}
 F(G)-F(H_{n,k})
 &=n(m_1+m_2)-2\sum_{U\in\calA_1}|U|-2\sum_{U\in\calA_2}|U|\\
 &\ge(b-2)m_1+(a-2)m_2\ge0.
 \end{aligned}
\]
Corollary~\ref{cor:completion-comparison} gives the required bound for $A(G)$ and its equality condition.

\smallskip\noindent\textbf{(iii)} If $G$ is complete, then $A(G)=A(K_{k+1})=A(H_{k+1,k})=\mu_{k+1,k}.$ Suppose that $3n\le4(k+1)$ and $G$ is not complete, and choose $v$ of degree $k$.
For any nonempty set $U$ such that $G[U]$ is disconnected, let $C$ be the vertex set of a component of $G[U]$ of minimum order.
Then $|C|\le|U|/2$, and for $u\in C$,
\[
 k\le d_G(u)\le |C|-1+n-|U|\le n-1-\frac{|U|}{2},
\]
Thus $|U|\le2(n-k-1)\le n/2$.
In particular, every $U\in\calA(G,v)$ satisfies this bound, so Lemma~\ref{lem:completion} gives $F(G)\ge F(H_{n,k})$.
Corollary~\ref{cor:completion-comparison} again gives the bound for $A(G)$ and its equality condition.
\end{proof}

\section{Conjecture and open questions}\label{sec:questions}

The complete graph minimizes the average order of connected induced subgraphs among connected simple graphs of fixed order and minimum degree at least three.
Prescribing a smaller minimum degree excludes this graph.
Theorem~\ref{thm:fixeddegree} suggests the following conjecture for this more restricted problem.

\begin{conjecture}
Let $n\ge k+1$ and $k\ge3$. Every connected simple graph $G$ of order $n$ with $\delta(G)=k$ satisfies
\[
 A(G)\ge A(H_{n,k})=\mu_{n,k},
\]
with equality if and only if $G\cong H_{n,k}$.
\end{conjecture}

Theorem~\ref{thm:fixeddegree} proves this conjecture for graphs with a cut vertex, for joins whose two parts each have at least two vertices, and for graphs satisfying $3n\le4(k+1)$.
It remains to consider 2-connected graphs satisfying $3n>4(k+1)$ that admit no such join decomposition. The proof for joins uses the comparison $F(G)\ge F(H_{n,k})$.
This inequality does not hold for all connected graphs: $G_{\mathrm c}$ has $n=7$, $k=3$, and $F(G_{\mathrm c})=14<F(H_{7,3})=18$.
Thus this comparison alone cannot prove the conjecture in full.

For graphs with a cut vertex, Theorem~\ref{thm:cutaverage} gives gaps above $n/2$ that are independent of the order, with equality only at orders seven and eight.
The minimum at other fixed orders remains to be determined.
For each $n\ge7$, what is the minimum of $A(G)$ among connected simple graphs of order $n$ with $\delta(G)\ge3$ that have a cut vertex but no cut edge, and which graphs attain it?
The corresponding question for graphs with a cut edge can be asked for each $n\ge8$.
One may also prescribe the minimum degree $\delta(G)=k$ in both problems.
The decomposition formulas in Subsection~\ref{sec:repair} keep track of the orders of the two rooted parts and may provide a starting point.

\subsection*{Data availability}
No external research data were used in this theoretical study.
Computer checks of finite graphs and algebraic identities were carried out during proof development.

\subsection*{Acknowledgements}
This research received no specific grant from any funding agency.

\subsection*{Use of artificial intelligence}
The authors developed the main tools of this paper by extending Haslegrave's work.
The Danus multi-agent workflow with GPT Astra was used to assist in developing some of the proofs.
GPT Sol 6.1 was used to revise and polish the language of the manuscript and to assist in preparing the TikZ figures.
The authors are responsible for the mathematical arguments, references, and final text.

\end{document}